\documentclass[reqno]{amsart}

\usepackage[latin1]{inputenc}
\usepackage{amsmath}
\usepackage{amsfonts}
\usepackage{amssymb}
\usepackage{fullpage}
\usepackage{hyperref}
\usepackage{enumitem}
\usepackage{todonotes}
\usepackage{cleveref}
\usepackage{breqn}
\setkeys{breqn}{breakdepth={1}}

\theoremstyle{plain}
\newtheorem{theorem}{Theorem}
\newtheorem{corollary}[theorem]{Corollary}
\newtheorem{lemma}[theorem]{Lemma}

\theoremstyle{definition}

\newtheorem{example}[theorem]{Example}

\theoremstyle{remark}

\begin{document}
    
    \author{Hieu D. Nguyen}
    \title{A Generalization of the Digital Binomial Theorem}
    \date{1-23-2015}

    \address{Department of Mathematics, Rowan University, Glassboro, NJ 08028.}
    \email{nguyen@rowan.edu}
    
    \subjclass[2010]{Primary 11}
    
    \begin{abstract}
        We prove a generalization of the digital binomial theorem by constructing a one-parameter subgroup of generalized Sierpinski matrices.  In addition, we derive new formulas for the coefficients of Prouhet-Thue-Morse polynomials and describe group relations satisfied by generating matrices defined in terms of these Sierpinski matrices. 
    \end{abstract} 

    \maketitle

\section{Introduction}
The classical binomial theorem describes the expansion of $(x+y)^N$ in terms of binomial coefficients, namely for any non-negative integer $N$, we have
\begin{equation}\label{eq:binomial-theorem}
(x+y)^N = \sum_{k=0}^N \binom{N}{k}x^k y^{N-k}, 
\end{equation}
where $\binom{N}{k}$ are defined in terms of factorials:
\[
\binom{N}{k}=\frac{N!}{k!(N-k)!}.
\]

In \cite{N1}, the author introduced a digital version of this theorem where the exponents appearing in (\ref{eq:binomial-theorem}) are viewed as sums of digits.  To illustrate this, consider the binomial theorem for $N=2$:
\begin{equation} \label{eq:binomial-theorem-n=2}
(x+y)^2  =x^2y^0+x^1y^1+x^1y^1+x^0y^2.
\end{equation}
It is easy to verify that (\ref{eq:binomial-theorem-n=2}) is equivalent to
\begin{equation} \label{eq:sum-of-digit-expansion-n=2}
(x+y)^{s(3)}  = x^{s(3)}y^{s(0)}+x^{s(2)}y^{s(1)}+x^{s(1)}y^{s(2)}+x^{s(0)}y^{s(3)},
\end{equation}
where $s(k)$ denotes the sum of digits of $k$ expressed in binary.  For example, $s(3)=s(1\cdot 2^1+1\cdot 2^0)=2$.  More generally, we have

\begin{theorem}[Digital Binomial Theorem \cite{N1}] \label{th:digital-binomial-theorem}
 Let $n$ be a non-negative integer.  Then
\begin{equation} \label{eq:digital-binomial-theorem-carry-free}
(x+y)^{s(n)}  = \sum_{\substack{0\leq m \leq n \\  (m,n-m) \ \textrm{carry-free}}}x^{s(m)}y^{s(n-m)}.
\end{equation}
\end{theorem}
\noindent Here, a pair of non-negative integers $(j,k)$ is said to be {\em carry-free} if their addition involves no carries when performed in binary.  For example, $(8,2)$ is carry-free since $8+2=(1\cdot 2^3+0\cdot 2^2+0\cdot 2^1+0\cdot 2^0)+1\cdot 2^1=10$ involves no carries.  It is clear that $(j,k)$ is carry-free if and only if $s(j)+s(k)=s(j+k)$ (see \cite{BEJ, N1}.  Also, observe that if $n=2^N-1$, then (\ref{eq:digital-binomial-theorem-carry-free}) reduces to (\ref{eq:binomial-theorem}). 

In this paper we generalize Theorem \ref{th:digital-binomial-theorem} to any base $b\geq 2$.  To state this result, we shall need to introduce a {\em digital dominance} partial order on $\mathbb{N}$ as defined by Ball, Edgar, and Juda in \cite{BEJ}.  Let
\begin{align*}
n & =n_0b^0+n_1b^1+\ldots + n_{N-1}b^{N-1}
\end{align*}
represent the base-$b$ expansion of $n$ and denote $d_i:=d_i(n)=n_i$ to be the $i$-th digit of $n$ in base $b$.  We shall say that $m$ is {\em digitally} less than or equal to $n$ if $m_i\leq n_i$ for all $i=0,1,\ldots, N-1$.  In that case, we shall write $m\preceq n$.  We are now ready to state our result.

\begin{theorem} \label{th:digital-binomial-theorem-non-binary}
Let $n$ be a non-negative integer.  Then
\begin{equation} \label{eq:digital-binomial-theorem-non-binary}
\prod_{i=0}^{N-1}\binom{x+y+d_i(n)-1}{d_i(n)} =\sum_{0\leq m\preceq n}\left[ \prod_{i=0}^{N-1}\binom{x+d_i(m)-1}{d_i(m)} \prod_{i=0}^{N-1}\binom{y+d_i(n-m)-1}{d_i(n-m)} \right].
\end{equation}
\end{theorem}

Let $\mu_j(n):=\mu_j^{(b)}(n)$ denote the multiplicity of the digit $j>0$ in the base-$b$ expansion of $n$, i.e.,
\[
\mu_j(n)=|\{i: d_i(n)=j\}|.
\]
As a corollary, we obtain

\begin{corollary} \label{co:digital-binomial-theorem-non-binary}
Let $n$ be a non-negative integer.  Then
\begin{align*}
\prod_{j=0}^{b-1}\binom{x+y+j-1}{j}^{\mu_j(n)} & =\sum_{0\leq m\preceq n}\left[ \prod_{j=1}^{b-1}\binom{x+j-1}{j}^{\mu_j(m)} \prod_{j=1}^{b-1}\binom{y+j-1}{j}^{\mu_j(n-m)} \right].
\end{align*}
\end{corollary}
\noindent Observe that if $b=2$, then Corollary \ref{co:digital-binomial-theorem-non-binary} reduces to Theorem \ref{th:digital-binomial-theorem}.

The source behind Theorem \ref{th:digital-binomial-theorem} is a one-parameter subgroup of Sierpinski matrices, investigated by Callan in \cite{C}, which encodes the digital binomial theorem.  To illustrate this, define a sequence of lower-triangular matrix functions $S_N(x)$ of dimension $2^N\times 2^N$ recursively by
\begin{equation}
S_1(x)=\left(
\begin{array}{rr}
1 & 0 \\
x & 1 \\
\end{array}
\right), \ \ S_{N+1}(x)=S_1(x)\otimes S_N(x),
\end{equation}
where $\otimes$ denotes the Kronecker product of two matrices.  For example, $S_2(x)$ and $S_3(x)$ can be computed as follows:
\begin{align*}
S_2(x) & =S_1(x)\otimes S_1(x) = \left(\begin{array}{rr} 1\cdot S_1(x) & 0\cdot S_1(x) \\ x\cdot S_1(x) & 1\cdot S_1(x) \end{array} \right)
= \left(
\begin{array}{rrrr}
1 & 0 & 0 & 0 \\
x & 1 & 0 & 0 \\
x & 0 & 1 & 0 \\
x^2 & x & x & 1
\end{array}
\right),
\\
\\
S_3(x) & =S_1(x)\otimes S_2(x) = \left(\begin{array}{rr} 1\cdot S_2(x) & 0\cdot S_2(x) \\ x\cdot S_2(x) & 1\cdot S_2(x) \end{array} \right)
=\left(
\begin{array}{llllllll}
1 & 0 & 0 & 0 & 0 & 0 & 0 & 0  \\
x & 1 & 0 & 0 & 0 & 0 & 0 & 0   \\
x & 0 & 1 & 0 & 0 & 0 & 0 & 0   \\
x^2 & x & x & 1 & 0 & 0 & 0 & 0  \\
x & 0 & 0 & 0 & 1 & 0 & 0 & 0  \\
x^2 & x & 0 & 0 & x & 1 & 0 & 0  \\
x^2 & 0 & x & 0 & x & 0 & 1 & 0  \\
x^3 & x^2 & x^2 & x & x^2 & x & x & 1
\end{array}
\right).
\end{align*}
Observe that when $x=1$, the infinite matrix $S=\lim_{N\rightarrow \infty}S_N(1)$ is of course Sierpinski's triangle.  

In \cite{C}, a formula for the entries of $S_N(x)=(\alpha_N(j,k,x))$, $0\leq j,k\leq 2^N-1$, is given in terms of the sum-of-digits function:
\begin{equation} \label{eq:formula-for-entries}
\alpha_N(j,k,x)=\left\{
\begin{array}{cl}
x^{s(j-k)}, & \mathrm{if} \ 0\leq k\leq j \leq 2^N-1 \ \mathrm{and} \  (k, j-k) \ \textrm{are} \ \textrm{carry-free}  \\ 
0, & \mathrm{otherwise}.
\end{array}
\right.
\end{equation}
Moreover, it was proven that $S_N(x)$ forms a one-parameter subgroup of $SL(2^N,\mathbb{R})$, i.e., the group of $2^N\times 2^N$ real matrices with determinant one.  Namely, we have
\begin{equation} \label{eq:one-parameter-matrix-product}
S_N(x)S_N(y)=S_N(x+y).
\end{equation}
If we denote the entries of $S_N(x)S_N(y)$ by $t_N(j,k)$, then the equality 
\[
t_N(j,k)=\sum_{i=0}^{2^N-1}\alpha_N(j,i,x)\alpha_N(i,k,y)=\alpha_N(j,k,x+y)\]
corresponds precisely to Theorem \ref{th:digital-binomial-theorem} with $j=0$ and $k=0$.  For example, if $N=2$, then (\ref{eq:one-parameter-matrix-product}) becomes
\begin{align*}
\left(
\begin{array}{llll}
1 & 0 & 0 & 0 \\
x & 1 & 0 & 0 \\
x & 0 & 1 & 0 \\
x^2 & x & x & 1 
\end{array}
\right)
\left(
\begin{array}{llll}
1 & 0 & 0 & 0 \\
y & 1 & 0 & 0 \\
y & 0 & 1 & 0 \\
y^2 & y & y & 1 
\end{array}
\right) & =
\left(
\begin{array}{cccc}
1 & 0 & 0 & 0 \\
x+y & 1 & 0 & 0 \\
x+y & 0 & 1 & 0 \\
(x+y)^2 & x+y & x+y & 1 
\end{array}
\right).
\end{align*}

The rest of this paper is devoted to generalizing Callan's construction of Sierpinski matrices to arbitrary bases and considering two applications of them.  In Section 2, we use these generalized Sierpinski matrices to prove Theorem \ref{th:digital-binomial-theorem-non-binary}.  In Section 3, we demonstrate how these matrices arise in the construction of Prouhet-Thue-Morse polynomials defined in \cite{N2}.  In Section 4, we describe a group presentation in terms of generators defined through these matrices and show that these generators satisfy a relation that generalizes E. Ferrand's result in \cite{F1}.


\section{Sierpinski Triangles}

To prove Theorem \ref{th:digital-binomial-theorem-non-binary}, we consider the following generalization of the Sierpinski matrix $S_N(x)$ in terms of binomial coefficients.  Define lower-triangular matrices $S_{b,N}(x)$ of dimension $b^N\times b^N$ recursively by
\begin{align*}
S_{b,1}(x) & =\left(
\begin{array}{cccll}
1 & 0 & 0 & \ldots & 0 \\
\binom{x}{1} & 1 & 0 & \ldots & 0 \\
\binom{x+1}{2} & \binom{x}{1} & 1 & \ldots & 0 \\
\vdots & \vdots & \vdots & \ddots  & \vdots \\
\binom{x+b-2}{b-1} & \binom{x+b-3}{b-2} & \binom{x+b-4}{b-3} & \ldots &  1
\end{array}
\right)=
\left\{
\begin{array}{cl}
\binom{x+j-k-1}{j-k}, & \textrm{if } 0\leq k \leq j \leq b-1 \\
0, & \textrm{otherwise}.
\end{array}
\right. \\
\end{align*}
and for $N>1$,
\[
S_{b,N+1}(x) =S_{b,1}(x)\otimes S_{b,N}(x).
\]

\begin{example} To illustrate our generalized Sierpinksi matrices, we calculate $S_{3,1}(x)$ and $S_{3,2}(x)$:
\begin{align*}
S_{3,1}(x)
& =\left(
\begin{array}{ccc}
1 & 0 & 0 \\
\binom{x}{1} & 1 & 0  \\
\binom{x+1}{2} & \binom{x}{1} & 1
\end{array} 
\right),
\\
S_{3,2}(x) & =S_{3,1}(x)\otimes S_{3,1}(x) \\
& =\left(
\begin{array}{ccccccccc}
1 & 0 & 0 & 0 & 0 & 0 & 0 & 0 & 0  \\
\binom{x}{1} & 1 & 0 & 0 & 0 & 0 & 0 & 0 & 0  \\
\binom{x+1}{2} & \binom{x}{1} & 1 & 0 & 0 & 0 & 0 & 0 & 0  \\
\binom{x}{1} & 0 & 0 & 1 & 0 & 0 & 0 & 0 & 0  \\
\binom{x}{1}\binom{x}{1} & \binom{x}{1} & 0 & \binom{x}{1} & 1 & 0 & 0 & 0 & 0  \\
\binom{x}{1}\binom{x+1}{2} & \binom{x}{1}\binom{x}{1} & \binom{x}{1} & \binom{x+1}{2} & \binom{x}{1} & 1 & 0 & 0 & 0  \\
\binom{x+1}{2} & 0 & 0 & \binom{x}{1} & 0 & 0 & 1 & 0 & 0  \\
\binom{x+1}{2}\binom{x}{1} & \binom{x+1}{2} & 0 & \binom{x}{1}\binom{x}{1} & \binom{x}{1} & 0 & \binom{x}{1} & 1 & 0 \\
\binom{x+1}{2}\binom{x+1}{2} & \binom{x+1}{2}\binom{x}{1} & \binom{x+1}{2} & \binom{x}{1}\binom{x+1}{2} & \binom{x}{1}\binom{x}{1} & \binom{x}{1} & \binom{x+1}{2} & \binom{x}{1} & 1 \\
\end{array}
\right).
\end{align*}
\end{example}

We now generalize Callan's result for $S_N(x)$ by presenting a formula for the entries of $S_{b,N}(x)$ (see \cite{IM} for a similar generalization of Callan's result but along a different vein).

\begin{theorem} \label{th:formula-for-entries-arbitrary-b}
Let $\alpha_N(j,k):=\alpha_N(j,k,x)$ denote the $(j,k)$-entry of $S_{b,N}(x)$.
Then
\begin{equation} \label{eq:formula-for-entries-arbitrary-b}
\alpha_N(j,k)=\left\{
\begin{array}{cl}
\binom{x+d_0-1}{d_0} \binom{x+d_1-1}{d_1} \cdots \binom{x+d_{N-1}-1}{d_{N-1}}, & \mathrm{if} \ 0\leq k\leq j \leq b^N-1 \ \mathrm{and} \  k \preceq j  \\ 
\\
0, & \mathrm{otherwise}.
\end{array}
\right.
\end{equation}
where $j-k=d_0b^0+d_1b^1+\ldots + d_Lb^L$ is the base-$b$ expansion of $j-k$, assuming $j\geq k$.
\end{theorem}

\begin{proof} We argue by induction on $N$.  It is clear that (\ref{eq:formula-for-entries-arbitrary-b}) holds for $S_{b,1}(x)$.  Next, assume that (\ref{eq:formula-for-entries-arbitrary-b}) holds for $S_{b,N}(x)$ and let $\alpha_{N+1}(j,k)$ be an arbitrary entry of $S_{b,N+1}(x)$, where $pb^N\leq j \leq (p+1)b^N-1$ and $qb^N \leq k \leq (q+1)b^N-1$ for some non-negative integers $p,q\in \{0,1,\ldots,b-1\}$.  Set $j'=j-pb^N$ and $k'=k-qb^N$.  We consider two cases:

\noindent Case 1. $p<q$.  Then $j\leq k$ and $\alpha_{N+1}(j,k)=0\cdot \alpha_N(j',k')=0$.

\noindent Case 2. $p\geq q$.  Then $j\geq k$ and
\begin{equation}\label{eq:case-2-part-1}
\alpha_{N+1}(j,k)=\binom{x+p-q-1}{p-q}\alpha_N(j',k').
\end{equation}
Let $j-k=d_0b^0+d_1b^1+\ldots + d_Nb^N$, where $d_N=p-q$.  Then $j'-k'=d_0b^0+d_1b^1+\ldots +d_{N-1}b^{N-1}$.  By assumption,
\begin{equation} \label{eq:case-2-part-2}
\alpha_{N}(j',k')=\left\{
\begin{array}{cl}
\binom{x+d_0-1}{d_0} \binom{x+d_1-1}{d_1} \cdots \binom{x+d_{N-1}-1}{d_{N-1}} & \mathrm{if} \ 0\leq k' \leq j' \leq b^N-1 \ \mathrm{and} \  k' \preceq j',  \\ 
\\
0 & \textrm{otherwise.}
\end{array}
\right.
\end{equation}
Since $k \preceq j$ if and only if $k' \preceq j'$, it follows from (\ref{eq:case-2-part-1}) and (\ref{eq:case-2-part-2}) that
\begin{equation}
\alpha_{N+1}(j,k)=\left\{
\begin{array}{cl}
\binom{x+d_0-1}{d_0} \binom{x+d_1-1}{d_1} \cdots \binom{x+d_N-1}{d_N} & \mathrm{if} \ 0\leq k\leq j \leq b^{N+1}-1 \ \mathrm{and} \  k \preceq j,  \\ 
\\
0 & \textrm{otherwise.}
\end{array}
\right.
\end{equation}
Thus,  (\ref{eq:formula-for-entries-arbitrary-b})  holds for $S_{b,N+1}$.
\end{proof}

Next, we show that $S_{b,N}(x)$ forms a one-parameter subgroup of $SL(b^N,\mathbb{R})$.  To prove this, we shall need the following two lemmas; the first is due to Gould \cite{G} and the second follows easily from the first through an appropriate change of variables.

\begin{lemma}[Gould \cite{G}] \label{le:binomial-sum-of-product}
\begin{equation}\label{eq:binomial-sum-formula}
\sum_{k=0}^n \binom{x+k}{k}\binom{y+n-k}{n-k} = \binom{x+y+n+1}{n}.
\end{equation}
\end{lemma}

\begin{proof} In \cite{G}, Gould derives (\ref{eq:binomial-sum-formula}) as a special case of a generalization of Vandemonde's convolution formula.  We shall proof  (\ref{eq:binomial-sum-formula}) more directly by presenting two different proofs, one using a combinatorial argument and the other using the beta function.

\vspace{5pt}
\noindent \textit{Combinatorial argument}. Let $A$, $B$, and $C=\{0,1,\ldots,n\}$ denote three sets containing $x$, $y$, and $n+1$ elements (all distinct), respectively, where $n$ is a non-negative integer.  For any non-negative integer $k$, define $A_k=A\cup \{0,\dots,k-1\}$ and $B_k=B\cup \{k+1,\dots,n\}$.  Then given any $n$-element subset $S$ of $A\cup B\cup C$, there exists a unique integer $k_S$ in $C-S$, called the index of $S$ with respect to $A$ and $B$, such that $|S\cap A_{k_S}|=k_S$ and $|S\cap B_{k_S}|=n-k_S$.  To see this, define $S_A=A\cap S$, $S_B=B\cap S$, and $T=C-S$.  We begin by deleting $|S_A|$ consecutive elements from $T$, in increasing order and beginning with its smallest element, to obtain a subset $T'$.  We then delete $|S_B|$ consecutive elements from $T'$, in decreasing order and beginning with its largest element, to obtain a subset $T''$, which must now contain a single element denoted by $k_S$.  It is now clear that $|S\cap A_{k_S}|=k_S$ and $|S\cap B_{k_S}|=n-k_S$.

To prove  (\ref{eq:binomial-sum-formula}), we count the $n$-element subsets $S$ of $A\cup B \cup C$ in two different ways.  On the one hand, since $|A\cup B \cup C|=x+y+n+1$, the number of such subsets is given by $\binom{x+y+n+1}{n}$.  On the other hand, we partition all such $n$-element subsets into equivalence classes according to each subset's index value.  Since $|S\cap A_{k}|=k$ and $|S\cap B_{k}|=n-k$, it follows that the number of $n$-element subsets $S$ having the same index $k$ is given by $\binom{x+k}{k}\binom{y+n-k}{n-k}$ and total number of $n$-element subsets is given by
\[
\sum_{k=0}^n \binom{x+k}{k}\binom{y+n-k}{n-k}.
\]
Lastly, we equate the two answers to obtain  (\ref{eq:binomial-sum-formula}).

\vspace{5pt}
\noindent \textit{Analytic argument}. Recall that the beta function is defined as $B(x,y)=\Gamma(x)\Gamma(y)/\Gamma(x+y)$, where $\Gamma(x)$ is the gamma function.  If $x$ is a non-negative integer, then $\Gamma(x+1)=x!$ and
\[
B(x,y)=\frac{(x-1)!(y-1)!}{(x+y-1)!}.
\]
For convenience, we shall write $x!$ to represent $\Gamma(x+1)$ even when $x$ is not non-negative.  We now divide  (\ref{eq:binomial-sum-formula}) by $n!x!y!/(x+y+n+1)!$ to obtain the following identity:
\begin{equation} \label{eq:beta-function-general}
\sum_{k=0}^n \binom{n}{n-k}B(x+k+1,y+n-k+1)=B(x+1,y+1).
\end{equation}
But (\ref{eq:beta-function-general}) is a generalization of the following classical property of the beta function:
\begin{equation} \label{eq:beta-function-property}
B(x,y-1)+B(x-1,y) =B(x-1,y-1).
\end{equation}
It is now straightforward to prove (\ref{eq:beta-function-general}) using an induction argument.
\end{proof}

\begin{lemma}\label{le:binomial-sum-of-product-2} Let $p$ and $q$ be positive integers with $q \leq p$.  Then
\begin{equation}\label{eq:binomial-sum-formula-2}
\sum_{v=q}^{p} \binom{x+p-v-1}{p-v} \binom{y+v-q-1}{v-q} =\binom{x+y+p-q-1}{p-q} .
\end{equation}
\end{lemma}

\begin{proof}
Set $w=v-q$.  Then (\ref{eq:binomial-sum-formula-2}) can be rewritten as
\[
\sum_{w=0}^{p-q} \binom{x+p-q-w-1}{p-q-w} \binom{y+w-1}{w} =\binom{x+y+p-q-1}{p-q},
\]
which follows from Lemma \ref{le:binomial-sum-of-product}.
\end{proof}

\begin{theorem} For all $N\in \mathbb{N}$, 
 \begin{equation}\label{eq:one-parameter}
 S_{b,N}(x)S_{b,N}(y)=S_{b,N}(x+y).
 \end{equation}
\end{theorem}

\begin{proof}  We argue by induction on $N$.  Lemma \ref{le:binomial-sum-of-product} proves that (\ref{eq:one-parameter}) holds for $S_{b,1}(x)$.  Next, assume that (\ref{eq:one-parameter}) holds for $S_{b,N}(x)$.   Let $t_{N+1}(j,k)$ denote the $(j,k)$-entry of $S_{b,N+1}(x)S_{b,N+1}(y)$.  Then
\begin{align*}
t_{N+1}(j,k) &=\sum_{m=0}^{b^{N+1}-1} \alpha_{N+1}(j,m,x)\alpha_{N+1}(m,k,y) \\
&=\sum_{v=0}^{b-1}\sum_{r=0}^{b^{N}-1} \alpha_{N+1}(j,vb^N+r,x)\alpha_{N+1}(vb^N+r,k,y).
\end{align*}
As before, assume $pb^N\leq j \leq (p+1)b^N-1$ and $qb^N \leq k \leq (q+1)b^N-1$ for some non-negative integers $p,q\in \{0,1,\ldots,b-1\}$.  Set $j'=j-pb^N$ and $k'=k-qb^N$.  Again, we consider two cases:
\\

\noindent Case 1: $j<k$.  Then $\alpha_{N+1}(j,k,x+y)=0$ by definition.  On the other hand, we have
\begin{align*}
t_{N+1}(j,k) & =\sum_{m=0}^{j} \alpha_{N+1}(j,m,x)\alpha_{N+1}(m,k,y)+\sum_{m=j+1}^{b^{N+1}-1} \alpha_{N+1}(j,m,x)\alpha_{N+1}(m,k,y) \\
& =\sum_{m=0}^{j} \alpha_{N+1}(j,m,x)\cdot 0+\sum_{m=j+1}^{b^{N+1}-1} 0\cdot \alpha_{N+1}(m,k,y) \\
& = 0
\end{align*}
and thus (\ref{eq:one-parameter}) holds.
\\

\noindent Case 2: $j\geq k$.  Since $S_{b,N+1}(x)  = S_{b,1}(x)\otimes S_{b,N}(x)$, we have
\[
\alpha_{N+1}(j,vb^N+r,x)=\left\{\begin{array}{cl}
\binom{x+p-v-1}{p-v}\alpha_{N}(j',r,x) & \textrm{if} \ j\geq vb^N+r, \\
0 & \textrm{if} \ j< vb^N+r. \\
\end{array}
\right.
\]
Similarly, we have
\[
\alpha_{N+1}(vb^N+r,k,y)=\left\{\begin{array}{cl}
\binom{y+v-q-1}{v-q}\alpha_{N}(r,k',y) & \textrm{if} \  k\leq vb^N+r, \\
0 & \textrm{if} \ k > vb^N+r. \\
\end{array}
\right.
\]
It follows that
\begin{align*}
t_{N+1}(j,k) 
&=\sum_{v=q}^{p}\sum_{r=0}^{b^{N}-1} \binom{x+p-v-1}{p-v} \binom{y+v-q-1}{v-q} \alpha_{N}(j',r,x)\alpha_{N}(r,k',y) \\
&=\left[\sum_{v=q}^{p} \binom{x+p-v-1}{p-v} \binom{y+v-q-1}{v-q}\right] \sum_{r=0}^{b^{N}-1} \alpha_{N}(j',r,x)\alpha_{N}(r,k',y) \\
& = \binom{x+y+p-q-1}{p-q}  \alpha_{N}(j',k',x+y) \\
& = \alpha_{N+1}(j,k,x+y),
\end{align*}
where we have made use of the inductive assumption and Lemma \ref{le:binomial-sum-of-product-2}.  This proves that (\ref{eq:one-parameter}) holds.
\end{proof}

As a corollary, we obtain Theorem \ref{th:digital-binomial-theorem-non-binary}, which we now prove.

\begin{proof}[Proof of Theorem \ref{th:digital-binomial-theorem-non-binary}]
Let $j=n$ and $k=0$.  Then the identity
\[
\sum_{m=0}^{b^{N}-1} \alpha_{N}(j,m,x)\alpha_{N}(m,k,y) = \alpha_{N}(j,k,x+y),
\]
which follows from (\ref{eq:one-parameter}), is equivalent to (\ref{eq:digital-binomial-theorem-non-binary}).
\end{proof} 

We end this section by describing the infinitesimal generator of $S_{b,N}(x)$.  Define $X_{b,1}(x)=(\chi_{j,k})$ to be a strictly lower-triangular matrix whose entries $\chi_{j,k}$ are given by
\begin{equation}
\chi_{j,k}=\left\{
\begin{array}{ll}
x/(j-k), & \textrm{if } j\geq k+1 \\
0, & \textrm{otherwise.} \\
\end{array}
\right.
\end{equation}
For $N>1$, we define matrices
\[X_{b,N+1}(x)=X_{b,1}(x)\oplus X_{b,N}(x)=X_{b,1}(x)\otimes I_{b^N} +I_b\otimes X_{b,N}(x),
\]
where $\oplus$ denotes the Kronecker sum and $I_{b^N}$ denotes the $b^N\times b^N$ identity matrix.  Observe that $X_{b,1}(x)$ has the following matrix form:
\begin{equation}
X_{b,1}(x)=\left(
\begin{array}{ccccc}
0 & 0 & 0 & \ldots & 0 \\
x & 0 & 0 & \ldots & 0 \\
x/2 & x & 0 & \ldots & 0 \\
x/3 & x/2 & x & \ldots & 0 \\
\vdots & \vdots & \vdots & \ddots & \vdots \\
x/(b-1) & x/(b-2) & x/(b-3) & \ldots & 0
\end{array}
\right)
\end{equation}

The following lemmas will be needed.  The first states a useful identity involving the unsigned Stirling numbers of the first kind, $c(n,k)$, defined by the generating function
\[
x(x+1)...(x+n-1)=\sum_{k=0}^n c(n,k) x^k.
\]
It is well known that $c(n,k)$ counts the number of $n$-element permutations consisting of $k$ cycles.

\begin{lemma} \label{le:stirling-number-formula}
Let $l$ and $n$ be positive integers with $l\geq n$.  Then
\begin{equation} \label{eq:stirling-number-formula}
\sum_{i=1}^{l-n+1}(i-1)!\binom{l}{i}c(l-i,n-1)=n c(l,n).
\end{equation}
\end{lemma}

\begin{proof} We give a combinatorial argument.  Let $A=\{1,2,...,l\}$.  We count in two different ways the number of permutations $\pi=\sigma_1\sigma_2 \cdots \sigma_n$ of $A$ consisting of $n$ cycles where we distinguish one of the cycles $\sigma_j$ of $\pi$.  On the one hand, since there are $c(l,n)$ such permutations $\pi$ and $n$ ways to distinguish a cycle of $\pi$, it follows that the answer is given by $n c(l,n)$.  On the other hand, we can construct $\pi$ by first choosing our distinguished cycle $\sigma_1$ consisting of $i$ elements.  The number of possibilities for $\sigma_1$ is $\binom{l}{i}(i-1)!$ since there are $\binom{l}{i}$ ways to choose $i$ elements from $A$ and $(i-1)!$ ways to construct a cycle from these $i$ elements.  It remains to construct the remaining cycles $\sigma_2,\cdots, \sigma_n$, which we view as a permutation $\pi'=\sigma_2\cdots \sigma_n$ on $l-i$ elements consisting of $n-1$ cycles.  Since there are $c(l-i,n-1)$ such possibilities for $\pi'$, it follows that the number of permutations $\pi$ with a distinguished cycle is given by 
\[\sum_{i=1}^{l-n+1}(i-1)!\binom{l}{i}c(l-i,n-1).
\]
Equating the two answers yields (\ref{eq:stirling-number-formula}) as desired.
\end{proof}

\begin{lemma} \label{le:infinitesimal-generator-power} Let $n$ be a positive integer with $1\leq n \leq b-1$.  Then
\begin{equation} \label{eq:infinitesimal-generator-power}
X_{b,1}^n(x)=(\chi_n(j,k)),
\end{equation}
where the entries $\chi_n(j,k)$ are given by
\begin{equation}  \label{eq:infinitesimal-generator-entry}
\chi_n(j,k)=\left\{
\begin{array}{cl}
 \frac{n!}{(j-k)!}c(j-k,n)x^n, & \textrm{if } j\geq k+n \\
0, & \textrm{otherwise.} \\
\end{array}
\right.
\end{equation}

\end{lemma}

\begin{proof} We argue by induction on $n$.  It is clear that (\ref{eq:infinitesimal-generator-entry}) holds when $n=1$.  Suppose $n>1$.  If $j< k+n$, then $\chi_n(j,k)=0$ because $X_{b,1}^m$ is a power of strictly lower-triangular matrices.  Therefore, assume $j\geq k+n$.  Then
\begin{align*}
\chi_n(j,k) & =\sum_{i=0}^{b-1}\chi_{n-1}(j,i)\chi_1(i,k) = \sum_{i=k+1}^{j-n+1}\chi_{n-1}(j,i)\chi_1(i,k)  \\
& =x^n \sum_{i=k+1}^{j-n+1}\frac{(n-1)!}{(j-i)!}c(j-i,n-1)\frac{1}{i-k} \\
& = \frac{(n-1)! x^n}{l!} \sum_{i=1}^{l-n+1}(i-1)!\binom{l}{i}c(l-i,n-1),
\end{align*}
where $l=j-k$.  It follows from Lemma \ref{le:stirling-number-formula} that
\begin{align*}
\chi_n(j,k)=\frac{(n-1)!x^n}{l!}\cdot n c(l,n)=\frac{n!}{(j-k)!}c(j-k,n)x^n
\end{align*}
as desired.
\end{proof}

\begin{lemma}  \label{le:infinitesimal-generator-base-case} We have
\begin{equation} \label{eq:infinitesimal-generator-base-case}
\mathrm{exp}(X_{b,1}(x))=S_{b,1}(x).
\end{equation}
\end{lemma}

\begin{proof} Denote the entries of $\mathrm{exp}(X_{b,1}(x))$ by $\xi(j,k)$.  It is clear that $\xi(j,k)=0$ for $j<k$ and $\xi(j,k)=1$ for $j=k$ since $X_{b,1}$ is strictly lower triangular.
Therefore, assume $j\geq k+1$.  Since $X_{b,1}^{n}=0$ for $n\geq b$, we have
\begin{align*}
\xi(j,k) & =\sum_{n=1}^{b-1}\frac{\chi_n(j,k)}{n!} \\
&  =\frac{1}{(j-k)!}\sum_{n=1}^{b-1}c(j-k,n) x^n \\
& =\frac{ x(x+1)\cdots (x+j-k-1) }{(j-k)!} \\
& =\binom{x+j-k-1}{j-k} \\
& = \alpha_1(j,k).
\end{align*}
Thus, (\ref{eq:infinitesimal-generator-base-case}) holds.
\end{proof}

\begin{theorem} Let $N$ be a positive integer.  Then
\begin{equation}\label{eq:infinitesimal-generator}
\mathrm{exp}(X_{b,N}(x))=S_{b,N}(x).
\end{equation}
\end{theorem}

\begin{proof} We argue by induction on $N$.
Lemma (\ref{le:infinitesimal-generator-base-case}) shows that (\ref{eq:infinitesimal-generator}) is true for $N=1$.  Then since $\mathrm{exp}(A\oplus B)=\mathrm{exp}(A)\otimes \mathrm{exp}(B)$ for any two matrices $A$ and $B$, it follows that
\[
\mathrm{exp}(X_{b,N}(x))= \mathrm{exp}(X_{b,1}(x) \oplus X_{b,N-1}(x))=S_{b,1}(x)\otimes S_{b,N-1}(x)=S_{b,N}(x),
\]
which proves (\ref{eq:infinitesimal-generator}).
\end{proof}


\section{Prouhet-Thue-Morse Polynomials}

In this section we demonstrate how the generalized Sierpinski matrices $S_{b,N}(1)$ arise in the study of Prouhet-Thue-Morse polynomials, first investigated by the author in \cite{N2}.   These polynomials were used in the same paper to give a new proof of the well-known Prouhet-Tarry-Escott problem, which seeks $b\geq 2$ sets of non-negative integers $S_0$, $S_1$, \dots, $S_{b-1}$ that have equal sums of like powers up to degree $M\geq 1$, i.e.,
\[
\sum_{n\in S_0} n^m =\sum_{n\in S_1} n^m= \cdots =\sum_{n\in S_{b-1}} n^m 
\]
for all $m=0,1,\dots,M$.  In 1851, E. Prouhet \cite{P} gave a solution (but did not publish a proof; see Lehmer \cite{L}) by partitioning the first $b^{M+1}$ non-negative integers into the sets $S_0,S_1,\ldots,S_{b-1}$ according to the assignment
\[
n\in S_{u_b(n)}.
\]
Here, $u_b(n)$ is the generalized Prouhet-Thue-Morse sequence, defined as the residue of  the sum of digits of $n$ (base $b$):
\[
u_b(n)=\sum_{j=0}^d n_j \mod b,
\]
where $n=n_db^d+\dots+n_0b^0$ is the base-$b$ expansion of $n$.  When $b=2$, $u(n):=u_2(n)$ generates the classical Prouhet-Thue-Morse sequence: $0,1,1,0,1,0,0,1,\ldots$. 

Let $A=(a_0,a_1,\dots,a_{b-1})$ be a \textit{zero-sum vector}, i.e., an ordered collection of $b$ arbitrary complex values that sum to zero:
\[
a_0+a_1+\dots+a_{b-1}=0.
\]
We define $F_N(x;A)$ to be the \textit{Prouhet-Thue-Morse} (PTM) polynomial of degree $b^N-1$ whose coefficients belong to $A$ and repeat according to $u_b(n)$, i.e.,
\begin{equation} \label{de:F}
F_N(x;A)=\sum_{n=0}^{b^{N}-1} a_{u_b(n)}x^n.
\end{equation}
In the case where $b=2$, $a_0=1$, and $a_1=-1$, we obtain the classic product generating function formula
\begin{equation} \label{eq:classic-product-generating-function}
 \prod_{m=0}^{N} (1-x^{2^m}) =\sum_{n=0}^{2^{N+1}-1} (-1)^{u(n)}x^n.
\end{equation}
Lehmer generalized this formula to the case where $A=(1,\omega,\omega^2,\ldots, \omega^{b-1})$ consists of all $b$-th roots of unity with $\omega=e^{i2\pi/b}$.  The following theorem, proven in \cite{N2}, extends this factorization to $F_N(x;A)$ for arbitrary zero-sum vectors.

\begin{theorem}[\cite{N2}] \label{th:factor-F}
Let $N$ be a positive integer and $A$ a zero-sum vector.  There exists a polynomial $P_N(x)$ such that
\begin{equation} \label{eq:ptm-polynomial}
F_N(x;A)=P_{N}(x)\prod_{m=0}^{N-1}(1-x^{b^m}).
\end{equation}
\end{theorem}
Theorem \ref{th:factor-F} is useful in that it allows us to establish that the polynomial $F_N(x,A)$ has a zero of order $N$ at $x=1$, from which Prouhet's solution follows easily by setting $N=M+1$ and differentiating $F_N(x;A)$ $m$ times (see \cite{N2}).  

We now derive formulas for the coefficients of $P_N(x)$ in terms of generalized Sierpinski triangles.  Towards this end, let
\[
P_N(x;C_N)=\sum_{n=0}^{b^N-1}c_nx^n
\]
denote a polynomial whose coefficients are given by the column vector $\mathbf{c}_N=(c_0,...,c_{b^N-1})^T$.  Also, let 
\[
\mathbf{a}_n=(a_{u_b(0)},a_{u_b(1)},\ldots,a_{u_b(b^N-1)})^T
\]
 be a column vector consisting of elements of $A$ generated by the PTM sequence $u_b(n)$.  Next, define a sequence of $b^N \times b^N$ matrices $M_N$ recursively by
\[
M_1=\left(
\begin{array}{rrrrrr}
1 & 0 & 0 & \ldots & 0 & 0 \\
-1 & 1 & 0 & \ldots & 0 & 0 \\
& & & \ddots \\
0 & 0 & 0 & \ldots & -1 & 1
\end{array}
\right)
\]
and for $N>1$,
\begin{equation}
M_{N+1}= M_1\otimes M_N =
\left(
\begin{array}{rrrrrr}
M_N & 0_N & 0_N &  \ldots & 0_N &  0_N \\
0_N & -M_N & M_N &  \ldots & 0_N & 0_N \\
& & &  \ddots  \\
0_N & 0_N  & 0_N & \ldots & -M_N & M_N \\
\end{array}
\right),
\end{equation}
where $M_1 \otimes M_N$ denotes the Kronecker product. 
The following theorem establishes a matrix relationship between the vectors $\mathbf{a}_N$ and $\mathbf{c}_N$.

\begin{theorem} \label{th:matrix-equation}
Let $A=(a_0,\ldots, a_{b-1})$ be a zero-sum vector.  The polynomial equation 
\begin{equation} \label{eq:polynomial-factorization}
F_N(x;A)=P_N(x;C_N)\prod_{m=0}^{N-1}(1-x^{b^m})
\end{equation}
 is equivalent to the matrix equation
\begin{equation}\label{eq:matrix-equation}
\mathbf{a}_N=M_N\mathbf{c}_N
\end{equation}
together with the condition $c_n=0$ for any $n$ that contains the digit $b-1$ in its base-$b$ expansion, where $0\leq n\leq b^N-1$.
\end{theorem}

To prove Theorem \ref{th:matrix-equation}, we shall need the following lemmas, which we state without proof since their results are easy to verify.

\begin{lemma} \label{le:system-equations-base-case}
Let $A=(a_0,a_1,\ldots, a_{b-1})$ and $C=\{c_0,\ldots,c_{b-1}\}$.  Then the polynomial equation
\[
\sum_{n=0}^{b-1}a_nx^n=\left(\sum_{n=0}^{b-1}c_nx^n\right)(1-x)
\]
is equivalent to the system of equations
\begin{align*}
a_0 & =c_0 \\
a_1 & =-c_0+c_1 \\
\cdots \\
a_{b-1} & =-c_{b-2} +c_{b-1}
\end{align*}
together with the condition $c_{b-1}=0$.
\end{lemma}

\begin{lemma} \label{le:matrix-equation-base-case}
The system of equations in Lemma \ref{le:system-equations-base-case}
can be expressed in matrix form as
\[
\mathbf{a}_1=M_1\mathbf{c}_1.
\]
\end{lemma}

\begin{proof}[Proof of Theorem \ref{th:matrix-equation}] We argue by induction on $N$.  It is clear that (\ref{eq:matrix-equation}) holds for $N=1$ since the polynomial equation $F_1(x;A)=P_1(x;C_1)(1-x)$ is equivalent to $\mathbf{a}_1=M_1\mathbf{c}_1$ because of Lemmas \ref{le:system-equations-base-case} and \ref{le:matrix-equation-base-case}.  Next, assume that (\ref{eq:matrix-equation}) holds for case $N$.  We shall prove that (\ref{eq:matrix-equation}) holds for case $N+1$.  Define 
\[
P_N(x;C_N(p))=\sum_{n=0}^{b^N-1}c_{n+pb^N}x^{n+pb^N}
\]
to be a polynomial with coefficient set $C_N(p)=\{c_{pb^N},\ldots,c_{(p+1)b^N-1}\}$.  We then expand the right-hand side of (\ref{eq:ptm-polynomial}) for case $N+1$ as follows:
\begin{align*}
P_{N+1}(x;C_{N+1})\prod_{m=0}^{N}(1-x^{b^m}) & = [P_N(x;C_N(0))+\ldots+P_N(x;C_N(b-1))]\left[\prod_{m=0}^{N-1}(1-x^{b^m})\right](1-x^{b^N}) \\
& = [Q_N(x;C_N(0))+\ldots+Q_N(x;C_N(b-1))](1-x^{b^N}) \\
&= Q_N(x;C_N(0))+[Q_N(x;C_N(1))-x^{b^N}Q_N(x;C_N(0))]+\ldots \\
& \ \ \ \ +[Q_N(x;C_N(b-1))-x^{b^N}Q_N(x;C_N(b-2))] - x^{b^N}Q_N(x;C_N(b-1)),
\end{align*}
where we define $Q_N(x;C_N(p))=P_N(x;C_N(p))(1-x^{b^N})$.
Equating this result with
\[
 F_{N+1}(x;A)=\sum_{n=0}^{b^{N+1}-1}a_{u_b(n)} x^n=F_N(x;A_N(0))+\ldots+F_N(x;A_N(b-1)),
 \]
 where
 \[
 F_N(x;A_N(p))=\sum_{n=0}^{b^N-1}a_{u(n+pb^N)}x^{n+pb^N},
 \]
 leads to the system of polynomial equations
\begin{align*}
F_N(x;A_N(0)) & =Q_N(x;C_N(0)) \\
F_N(x;A_N(1)) & =Q_N(x;C_N(1))-x^{b^N}Q_N(x;C_N(0)) \\
& \ldots \\
F_N(x;A_N(b-1)) & =Q_N(x;C_N(b-1))-x^{b^N}Q_N(x;C_N(b-2))
\end{align*}
together with the condition $Q_N(x;C_N(b-1))=0$, i.e. $c_n=0$ for all $(b-1)b^N \leq n \leq b^{N+1}-1$.
It follows from Lemmas \ref{le:system-equations-base-case} and \ref{le:matrix-equation-base-case} that this system is equivalent to the system of matrix equations
\begin{align*}
\mathbf{a}_N(0) & =M_N\mathbf{c}_N(0) \\
\mathbf{a}_N(1) & = -M_N\mathbf{c}_N(0) + M_N\mathbf{c}_N(1) \\
& \ldots \\
\mathbf{a}_N(b-1) & = -M_N\mathbf{c}_N(b-2) + M_N\mathbf{c}_N(b-1),
\end{align*}
where the first matrix equation by assumption satisfies the condition $c_n=0$ for any $0\leq n\leq b^N-1$ that contains the digit $b-1$ in its base-$b$ expansion.
This in turn is equivalent to the matrix equation $\mathbf{a}_{N+1}=M_{N+1}\mathbf{c}_{N+1}$ together with the condition that $c_n=0$ for any $0\leq n \leq b^{N+1}-1$ that contains the digit $b-1$ in its base-$b$ expansion.  This establishes the theorem for case $N+1$ and completes the proof.
\end{proof}

\begin{lemma}
The matrix $M_N$ has inverse $S_N=M_N^{-1}$, where $S_N$ is given recursively by
\begin{equation}
S_1=
\left(
\begin{array}{rrrrr}
1 & 0 & 0 &  \ldots & 0 \\
1 & 1 & 0 &  \ldots & 0 \\
& & &  \ddots  \\
1 & 1  & 1 & \ldots & 1 \\
\end{array}
\right)
\end{equation}
and for $N>1$,
\begin{equation}
S_{N+1}=S_1\otimes S_N=
\left(
\begin{array}{rrrrr}
S_N & 0_N & 0_N &  \ldots &  0_N \\
S_N & S_N & 0_N &  \ldots &  0_N \\
& & &  \ddots  \\
S_N & S_N  & S_N & \ldots & S_N \\
\end{array}
\right).
\end{equation}
Thus, if $\mathbf{a}_N=M_N\mathbf{c}_N$, then
\begin{equation} \label{eq:matrix-equation-case-N}
\mathbf{c}_N=S_N\mathbf{a}_N.
\end{equation}
\end{lemma}

\begin{proof}
It is straightforward to verify directly that $S_1=M_1^{-1}$.  Since $S_{N+1}=S_1\otimes S_{N}$ and $M_{N+1}=M_1\otimes M_N$, it follows that $S_N=M_N^{-1}$.
  \end{proof}

Observe that $S_N=S_{b,N}(1)$ is the generalized Sierpinski triangle defined in the previous section. 
We now present a formula for the coefficients $c_n$ in terms of the elements of $A$.

\begin{theorem} \label{th:formula-coefficients}
Let $A=\{a_0,\ldots,a_{b-1}\}$ be a zero-sum vector.  Then the polynomial equation (\ref{eq:polynomial-factorization}) has solution $P_N(x;C_N)$ whose coefficients $c_n$, $0\leq n \leq b^N-1$, are given by
\begin{equation} \label{eq:formula-coefficients}
c_n=\sum_{k\in I_b(n)}a_{u_b(k)},
\end{equation}
where $I_b(n)=\{k\in \mathbb{N}: k\preceq n \}$.  Moreover, $c_n=0$ for all $n$ whose base-$b$ expansion contains the digit $b-1$.
\end{theorem}

\begin{proof} From Theorem \ref{th:formula-for-entries-arbitrary-b}, we know that the non-zero entries in the $n$-th row of $S_N$, which are all equal to 1, are located at $(n,k)$ where $k\preceq n$.  Formula (\ref{eq:formula-coefficients}) now follows from (\ref{eq:matrix-equation-case-N}).  It remains to show that (\ref{eq:formula-coefficients}) yields $c_n=0$ for all $n$ whose base-$b$ expansion contains the digit $b-1$.  Let $n=n_0b^0+\ldots+n_Lb^L+\ldots + n_{N-1}b^{N-1}$ where $n_L=b-1$.  Denote $I_b(n;L)$ to be the subset of $I_b(n)$ consisting of integers whose base-$b$ expansion has digit 0 at position $L$, i.e.,
\[
I_b(n;L)=\{k\in \mathbb{N}:k\preceq n, k=k_0b^0+\ldots+\ldots+k_Nb^N, k_L=0\}.
\]
Then using the fact that $A$ is a zero-sum vector, i.e., $a_0+\ldots + a_{b-1}=0$, we have
\begin{align*}
c_n & =\sum_{k\in I_b(n;L)}a_{u_b(k)} + \sum_{k\in I_b(n;L)}a_{u_b(k+b^L)}+\ldots+\sum_{k\in I_b(n;L)}a_{u_b(k+(b-1)b^L))} \\
& =\sum_{k\in I_b(n;L)}[a_{u_b(k)} + a_{(u_b(k)+1)_b}+\ldots+a_{(u_b(k)+(b-1))_b}] \\
& = 0
\end{align*}
as desired.
\end{proof}

Next, we specialize Theorem \ref{th:formula-coefficients} to base $b=3$.  Define $w(n)$ to be the sum of the digits of $n$ in its base-$3$ representation modulo 2.

\begin{corollary} Suppose $b=3$.  Let $n\in \mathbb{N}$ be such that its base-3 representation does not contain the digit 2.  Then 
\begin{equation} \label{eq:formula-coefficient-base-3}
c_n=(-1)^{w(n)}a_{u_3(2n)}.
\end{equation}
\end{corollary}

\begin{proof} Let $n=n_03^0+\ldots+n_{N-1}3^{N-1}$ be the base-3 representation of $n$ with no digit equal to 2 and leading digit $n_{N-1}=1$.  We argue by induction on $N$.  It is clear that (\ref{eq:formula-coefficient-base-3}) holds when $N=1$ since from (\ref{eq:formula-coefficients}) we have
\[
c_1=a_0+a_1=-a_2=(-1)^{w(n)}a_{u_3(2n)}.
\]
Next, assume (\ref{eq:formula-coefficient-base-3}) holds for a given $N$.  To prove that (\ref{eq:formula-coefficient-base-3}) holds for $N+1$, we consider two cases by decomposing $n=n'+3^{N}$.
\\

\noindent Case 1: $n_i=0$ for all $i=0,\ldots,N-1$.  Then $n=3^n$ and
\begin{align*}
c_{n} & =\sum_{k\in I_3(n)}a_{u_3(k)}=a_{u_3(0)}+a_{u_3(n)}=a_0+a_1 = -a_2 \\
& = (-1)^{w(n)}a_{u_3(2n)}.
\end{align*}

\noindent Case 2: $n_L=1$ for some $0\leq L \leq N-1$.  Then set $n''=n'-3^L$.  It follows that
\begin{align*}
c_{n} & =\sum_{k\in I_3(n)}a_{u_3(k)}=\sum_{k\in I_3(n,L)}a_{u_3(k)}+\sum_{k\in I_3(n,L)}a_{u_3(k+3^N)} \\
& =\sum_{k\in I_3(n')}a_{u_3(k)}+\sum_{k\in I_3(n')}a_{u_3(k+3^L)} \\
& =\sum_{k\in I_3(n')}a_{u_3(k)}+\sum_{k\in I_3(n',L)}a_{u_3(k+3^L)}+\sum_{k\in I_3(n',L)}a_{u_3(k+2\cdot 3^L)} +\sum_{k\in I_3(n',L)}a_{u_3(k)}-\sum_{k\in I_3(n',L)}a_{u_3(k)} \\
& =\sum_{k\in I_3(n')}a_{u_3(k)}+\sum_{k\in I_3(n',L)}[a_{u_3(k)}+a_{(u_3(k)+1)_3} +a_{(u_3(k)+2)_3}]-\sum_{k\in I_3(n'')}a_{u_3(k)} \\
& = (-1)^{w(n')}a_{u_3(2n')}- (-1)^{w(n'')}a_{u_3(2n'')} \\
& = (-1)^{w(n')}a_{u_3(2n')}+(-1)^{w(n''+3^L)}a_{u_3(2(n'-3^L))} \\
& = (-1)^{w(n')}[a_{u_3(2n')}+a_{(u_3(2n')-2)_3} ] = (-1)^{w(n')}[-a_{(u_3(2n')-1)_3} ]\\
& = (-1)^{w(n'+3^N)}a_{(u_3(2n')+2)_3))} = (-1)^{w(n)}a_{u_3(2n'+2\cdot 3^L)} \\
& = (-1)^{w(n)}a_{u_3(2n)}.
\end{align*}
Thus, (\ref{eq:formula-coefficient-base-3}) holds for $N+1$.
\end{proof}


\section{Group Generators and Relations}

In this section, we describe group generators and relations defined by the matrices $S_N$ and $M_N$.  Recall that $S_N$ and $M_N$ are matrices of dimension $b^N\times b^N$ for a given base $b$.  Define $T_N=M_N^t$ to be the transpose of $M_N$ and $U_N=S_NT_N$, $V_N=T_NS_N$.  
The following lemma gives a recursive construction of $U_N$ and$V_N$.

\begin{lemma} We have
\begin{align*}
U_{N+1} & =U_1\otimes U_N \\
V_{N+1} & = V_1\otimes V_N
\end{align*}
\end{lemma}

\begin{proof} The result follows from the mixed-product property of the Kronecker product.  We demonstrate this for $U_{N+1}$:
\begin{align*}
U_{N+1}=S_{N+1}T_{N+1}=(S_1\otimes S_N)(T_1\otimes T_N)=(S_1T_1)\otimes (S_NT_N)=U_1\otimes U_N.
\end{align*}
The calculation is the same for $V_{N+1}$.
\end{proof}

Observe that $U_N=(u_{i,j})$ and $V_N=(v_{i,j})$ are skew-triangular and that $V_N$ is the skew-transpose of $U_N$, i.e., $v_{i,j}=u_{b-1-j,b-1-i}$ for $i,j=0,1,\ldots,b-1$.  Here are some examples of $U_N$ and $V_N$ when $b=2$:
\begin{equation}
U_1=\left(
\begin{array}{rr}
1 & -1 \\
1 & 0
\end{array}
\right), \ \
U_2=\left(
\begin{array}{rrrr}
1 & -1 & -1 & 1 \\
1 & 0 & -1 & 0 \\
1 & -1 & 0 & 0 \\
1 & 0 & 0 & 0
\end{array}
\right), \ \
U_3=\left(
\begin{array}{rrrrrrrr}
1 & -1 & -1 & 1 & -1 & 1 & 1 & -1 \\
1 & 0 & -1 & 0 &-1 & 0 & 1 & 0 \\
1 & -1 & 0 & 0 & -1 & 1 & 0 & 0 \\
1 & 0 & 0 & 0 & -1 & 0 & 0 & 0 \\
1 & -1 & -1 & 1 & 0 & 0 & 0 & 0 \\
1 & 0 & -1 & 0 & 0 & 0 & 0 & 0 \\
1 & -1 & 0 & 0 & 0 & 0 & 0 & 0 \\
1 & 0 & 0 & 0 & 0 & 0 & 0 & 0
\end{array}
\right)
\end{equation}

\begin{equation}
V_1=\left(
\begin{array}{rr}
0 & -1 \\
1 & 1
\end{array}
\right), \ \
V_2=\left(
\begin{array}{rrrr}
0 & 0 & 0 & 1 \\
0 & 0 & -1 & -1 \\
0 & -1 & 0 & -1 \\
1 & 1 & 1 & 1
\end{array}
\right), \ \
V_3=\left(
\begin{array}{rrrrrrrr}
0 & 0 & 0 & 0 & 0 & 0 & 0 & -1 \\
0 & 0 & 0 & 0 & 0 &  0 & 1 & 1 \\
0 & 0 & 0 & 0 & 0 & 1 & 0 & 1 \\
0 & 0 & 0 & 0 & -1 & -1 & -1 & -1 \\
0 & 0 & 0 & 1 & 0 & 0 & 0 & 1 \\
0 & 0 & -1 & -1 & 0 & 0 & -1 & -1\\
0 & -1 & 0 & -1 & 0 & -1 & 0 & -1\\
1 & 1 & 1 & 1 & 1 & 1 & 1 & 1
\end{array}
\right).
\end{equation}

The next lemma establishes that the eigenvalues of $U_N$ and $V_N$ are $(b+1)$-th roots of 1 or $-1$.

\begin{lemma} \label{le:eigenvalues}
The set of eigenvalues of $U_1$ and $V_1$ are exactly the same and consist of all roots of the polynomial equation
\[
-1+r-r^2+\ldots+(-1)^{b+1} r^b=0.
\]
\end{lemma}

\begin{proof}
Let $r$ be a root of $-1+r-r^2+\ldots+(-1)^{b+1} r^b=0$.  We claim that $r$ is an eigenvalue of $U_1$ with eigenvector $\mathbf{v}=(v_1,\ldots, v_b)^T$, where 
\[
v_k=\left\{
\begin{array}{cl}
\sum_{j=1}^k (-1)^{j+1}r^j, & 1\leq k \leq b-1 \\
0, & k=b.
\end{array}
\right.
\]
Denote $\mathbf{w}=(U_1-rI_b)\mathbf{v}=\{w_1,\ldots,w_b\}$.  It suffices to show $\mathbf{w}=0$.  Assume $b=2$.  We have
\begin{align*}
w_1 & =(1-r)r-1=-1+r-r^2 =0 \\
w_2 & =r-r=0.
\end{align*}
For $b>2$, we have
\begin{align*}
w_1 & =(1-r)r-(r-r^2)=0 \\
w_2 & =r-r(r-r^2)-(r-r^2+r^3)=0 \\
& \ldots \\
w_{b-1} & = r - r(r-r^2+r^3+\ldots+r^{b-1})-1=-1+r-r^2+\ldots +(-1)^{b+1}r^b = 0 \\
w_b & =r-r=0.
\end{align*}
Thus, $\mathbf{w}=0$.  It can be shown by a similar argument that the eigenvalues of $V_1$ are exactly the same as those of $U_1$.
\end{proof}

\begin{theorem} \label{th:group-relation}
The matrices $U_N$ and $V_N$ satisfy the relation
\begin{equation}\label{eq:group-relation}
U_N^{b+1}=V_N^{b+1}=(-1)^{N(b+1)}I_{b^N},
\end{equation}
where $I_{b^N}$ is the $b^N\times b^N$ identity matrix.
\end{theorem}

\begin{proof}
We shall only prove (\ref{eq:group-relation}) for $U_N$ since the proof for $V_N$ is the same.  We argue by induction on $N$.  When $N=1$, we know from Lemma \ref{le:eigenvalues} that all eigenvalues of $U_1$ are roots of 
$-1+r-r^2+\ldots+(-1)^{b+1} r^b=0$.
It follows that every eigenvalue $r$ satisfies $r^{b+1}=(-1)^{b+1}$ and since they are all distinct, the corresponding eigenvectors are all linearly independent.  Thus, $U_1^{b+1}=(-1)^{b+1}I_b$.

Next, assume that (\ref{eq:group-relation}) holds for $U_{N-1}$.  It follows from the mixed-product property of the Kronecker product that
\begin{align*}
U_N^{b+1} & = (U_1\otimes U_{N-1})^{b+1} = U_1^{b+1}\otimes U_{N-1}^{(N-1)(b+1)} \\
& =(-1)^{b+1}I_b\otimes (-1)^{(N-1)(b+1)}I_{b^{N-1}} \\
& =(-1)^{N(b+1)}I_{b^N}.
\end{align*}
Thus,  (\ref{eq:group-relation}) holds for $N$.
\end{proof}

We note that for $b=2$,  Ferrand proved in \cite{F1} that the matrices $S_N$ and $T_N$ satisfy the 3-strand braid relation
\begin{equation} \label{eq:braid-relation}
S_NT_NS_N=T_NS_NT_N.
\end{equation}
We give an alternate proof of (\ref{eq:braid-relation}) based on Theorem \ref{th:group-relation}.  Define $Q_N=S_NT_NS_N$ and $R_N=T_NS_NT_N$.  Then $Q_NR_N=U_N^3=(-1)^{3N}I_{2^N}=(-1)^NI_{2^N}$ because of (\ref{eq:group-relation}).  Moreover, we claim that
 $Q_N^2=R_N^2=(-1)^NI_{2^N}$.  This follows by induction on $N$, which we demonstrate for $Q_N$ by again using the mixed-product property of the Kronecker product:
 \begin{align*}
Q_N^2 &=(S_NT_NS_N^2T_NS_N) \\
& = (S_1T_1S_1^2T_1S_1)\otimes(S_{N-1}T_{N-1}S_{N-1}^2T_{N-1}S_{N-1}) \\
& = (-I_2)\otimes((-1)^{N-1}I_{2^{N-1}}) \\
& = (-1)^N I_{2^N}
\end{align*}
Thus, $Q_NR_N=Q_N^2$, which implies the braid relation $Q_N=R_N$.   However, we find that the braid relation fails to hold for $b>2$. 


\end{document}